\documentclass[]{interact}

\usepackage{epstopdf}
\usepackage[caption=false]{subfig}

\usepackage[numbers,sort&compress]{natbib}
\bibpunct[, ]{[}{]}{,}{n}{,}{,}
\makeatletter
\def\NAT@def@citea{\def\@citea{\NAT@separator}}
\makeatother

\theoremstyle{plain}
\newtheorem{theorem}{Theorem}[section]

\theoremstyle{definition}

\theoremstyle{remark}

\usepackage{amsmath}    
\usepackage{amsfonts}   
\usepackage{amssymb}    
\usepackage{mathrsfs}   
\usepackage{graphicx}   
\usepackage{color}
\usepackage{algorithm}
\usepackage[noend]{algpseudocode}
\usepackage{multirow}
\usepackage{url}

\newcommand{\ie}{{\it i.e.}}

\newcommand{\eg}{{\it e.g.}}

\def\cJ{\mathcal J}
\def\cK{\mathcal K}

\def\cN{\mathcal N}
\def\cP{\mathcal P}
\def\cQ{\mathcal Q}

\newcommand{\bone}{{\bf 1}}

\newcommand{\bbE}{{\mathbb E}}

\newcommand{\bbR}{{\mathbb R}}

\newcommand{\pr}{\mathbb{P}}

\newcommand{\bc}{\begin{center}}
\newcommand{\ec}{\end{center}}
\newcommand{\be}{\begin{equation}}
\newcommand{\ee}{\end{equation}}
\newcommand{\ba}{\begin{array}}
\newcommand{\ea}{\end{array}}
\newcommand{\bean}{\setlength\arraycolsep{2pt}\begin{eqnarray*}}
\newcommand{\eean}{\end{eqnarray*}}
\newcommand{\bea}{\setlength\arraycolsep{2pt}\begin{eqnarray}}
\newcommand{\eea}{\end{eqnarray}}
\newcommand{\ben}{\begin{enumerate}}
\newcommand{\een}{\end{enumerate}}
\newcommand{\bed}{\begin{itemize}}
\newcommand{\eed}{\end{itemize}}

\DeclareMathOperator*{\argmax}{argmax}
\DeclareMathOperator*{\argmin}{argmin}

\def\half{\hbox{$1\over2$}}

\begin{document}

\articletype{}

\title{An EM based Iterative Method for Solving	Large Sparse Linear Systems}

\author{
\name{Minwoo Chae\textsuperscript{a}\thanks{CONTACT Minwoo Chae. Email: minwooo.chae@gmail.com} and Stephen G. Walker\textsuperscript{b}}
\affil{\textsuperscript{a}Department of Mathematics, Applied Mathematics and Statistics, Case Western Reserve University, Cleveland, OH, USA; \textsuperscript{b}Department of Mathematics, The University of Texas at Austin, Austin, TX, USA}
}

\maketitle

\begin{abstract}
We propose a novel iterative algorithm for solving a large sparse linear system. The method is based on the EM algorithm. If the system has a unique solution, the algorithm guarantees convergence with a geometric rate. Otherwise, convergence to a minimal Kullback--Leibler divergence point is guaranteed. The algorithm is easy to code and competitive with other iterative algorithms.
\end{abstract}


\section{Introduction} \label{sec:intro}

An important problem is to find a solution to a system of linear equations
\be\label{eq:ls}
	Ax = b,
\ee
where $A = (a_{ij})$ is an $m_1 \times m_2$ matrix and $b$ is an $m_1$-dimensional vector.
We mainly consider the square matrix where $m_1 = m_2 = m$, but the theory and computations presented in the paper works for general $A$.
If $A$ is nonsingular with inverse matrix $A^{-1}$, there exists a unique solution to \eqref{eq:ls}, denoted by $x^* = A^{-1} b$.
When the dimension $m$ is large, however, finding the inverse matrix $A^{-1}$ is computationally unfeasible.
As alternatives, a number of iterative methods have been proposed to find a sequence $(x_n)$ approximating $x^*$, and they are often implementable when $A$ is sparse, that is, most $a_{ij}$'s are zero.
For reviews of these iterative methods within a unified framework, we refer to the monograph \cite{saad2003iterative}.
Many sources of such large sparse linear systems come from the discretization of a partial differential equation; see Chapter 2 of \cite{saad2003iterative}.
Within statistical applications, sparse design matrices have been considered in \cite{kennedy1980gentle, koenker2003sparsem} and an algorithm sampling high-dimensional Gaussian random variables with sparse precision matrices has been developed in \cite{aune2013iterative}.

We start with a brief introduction of the most widely used iterative methods for solving \eqref{eq:ls}.
Current iterative methods are coordinate-wise updating algorithms.
Two of the most well-known methods are \emph{Jacobi} and \emph{Gauss--Seidel} which can be found in most standard textbooks.
Given $x_n = (x_{n,j})$, the Jacobi and Gauss--Seidel methods update $x_{n+1}$ via
\bean
	x_{n+1,j} &=& \frac{1}{a_{jj}} \left(b_j - \sum_{i\neq j} a_{ji} \,x_{n,i} \right)
\eean
and
\bean
	x_{n+1,j} &=& \frac{1}{a_{jj}} \left(b_j - \sum_{i=1}^{j-1} a_{ji} \, x_{n+1,i} - \sum_{i=j+1}^{m} a_{ji} \,x_{n,i}\right),
\eean
respectively.
Although they are simple and convenient, both of them are restrictive in practice because $(x_n)$ is not generally guaranteed to converge to $x^*$; see Chapter 4 of \cite{saad2003iterative}.

The \emph{Krylov subspace methods}, which are based on the \emph{Krylov subspace} of $\bbR^m$,
$$
	\cK_n = {\rm span} \Big\{r_0, Ar_0, A^2 r_0, \ldots, A^{n-1}r_0 \Big\},
$$
are the dominant approaches, where $x_0$ is an initial guess and $r_0 = b-Ax_0$.
Under the assumption that $A$ is sparse, matrix-vector multiplication is cheap to compute, so it is not difficult to handle $\cK_n$ even when $m$ is very large.
If $A$ is symmetric and positive definite (SPD), the standard choice for solving \eqref{eq:ls} is the \emph{conjugate gradient method} (CG; \cite{hestenes1952methods}).
This is an \emph{orthogonal projection method}, see Chapter 5 of \cite{saad2003iterative}, onto $\cK_n$, finding $x_n \in x_0 + \cK_n$ such that $b - A x_n \perp \cK_n$.
To be more specific, recall that two vectors $u,v\in\bbR^m$ are called \emph{$A$-conjugate} if $u^T A v = 0$.
If $A$ is symmetric and positive definite, then this quadratic form defines an inner product, and there is a basis for $\bbR^m$ consisting of mutually $A$-conjugate vectors.
The CG method sequentially generates mutually $A$-conjugate vectors $p_1, p_2, \ldots$, and approximates $x^* = \sum_{j=1}^m \alpha_j p_j$ as $x_n = \sum_{j=1}^n \alpha_j p_j$, where $\alpha_j = p_j^T b / p_j^T A p_j$.
Using the symmetry of $A$, the computation can be simplified as in Algorithm \ref{alg:cg}.
Here $\|\cdot\|_q$ denotes the $\ell_q$-norm on $\bbR^m$.

\begin{algorithm}
\caption{Conjugate gradient method for SPD $A$} \label{alg:cg}
\begin{algorithmic}[1]
\State {\bf Input}: $A, b, x_0$ and $\epsilon_{\rm tol} > 0$
\State $j\gets 0$
\State $r_0 \gets b - A x_0$
\State $p_0 \gets r_0$
\While{$\|r_j\|_2 > \epsilon_{\rm tol}$}
	\State $\alpha_j \gets r_j^T r_j / p_j^T Ap_j$ \label{step:cg1}
	\State $x_{j+1} \gets x_j + \alpha_j p_j$
	\State $r_{j+1} \gets r_j - \alpha_j Ap_j$
	\State $\beta_j \gets r_{j+1}^T r_{j+1} / r_j^T r_j$
	\State $p_{j+1} \gets r_{j+1} + \beta_j p_j$ 
	\State $j \gets j+1$\label{step:cg2}
\EndWhile
\State \Return $x_j$
\end{algorithmic}
\end{algorithm}

For a general matrix $A$, the \emph{generalized minimal residual method} (GMRES; \cite{saad1986gmres}) is the most popular.
It is an \emph{oblique projection method}, see Chapter 5 of \cite{saad2003iterative}, which finds $x_k \in \cK_k$ satisfying $b - A x_k \perp A\cK_k$, where $A\cK_k = \{Av: v\in\cK_k\}$.
When implementing GMRES, \emph{Arnoldi's method} \cite{arnoldi1951principle} is applied for computing an orthonormal basis of $\cK_k$.
The method can be written as in Algorithm \ref{alg:gmres}.
For a given initial $x_0$, let us write the result of Algorithm \ref{alg:gmres} as $G_k(x_0)$.
Since the computational cost of Algorithm \ref{alg:gmres} is prohibitive for large $k$, a restart version of GMRES$(k)$, defined as $x_{n+1} = G_k(x_n)$, is applied with small $k$.
It should be noted that the generalized conjugate residual (GCR; \cite{elman1982iterative}), ORTHODIR \cite{young1980generalized} and Axelsson's method \cite{axelsson1980conjugate} are mathematically equivalent to GMRES; but it is known in \cite{saad1986gmres} that GMRES is computationally more efficient and reliable.
Further connections between these methods are discussed in \cite{saad1985conjugate}.
Convergence is guaranteed, but there are restrictions; see Section \ref{sec:comparison}.

\begin{algorithm}
\caption{GMRES$(k)$} \label{alg:gmres}
\begin{algorithmic}[1]
\State {\bf Input}: $A, b, x_0$ and $\epsilon_{\rm tol} > 0$
\State $\beta \gets \|b - A x_0\|_2$
\State $v_1 \gets (b-Ax_0)/\beta$
\For{$j=1, \ldots, k$} \label{step:gmres0}
	\State $w_j \gets Av_j$ \label{step:gmres1}
	\For {$i=1, \ldots, j$}
		\State $h_{ij} \gets w_j^T v_i$ \label{step:gmres2}
		\State $w_j \gets w_j - h_{ij} v_i$ \label{step:gmres3}
	\EndFor
	\State $h_{j+1,j} \gets \|w_j\|_2$ \label{step:gmres4}
	\If {$h_{j+1,j} < \epsilon_{\rm tol}$ } set $k\gets j$ and {\bf break}\EndIf
	\State $v_{j+1} \gets w_j / h_{j+1,j}$ \label{step:gmres5}
\EndFor
\State $y_k \gets \argmin_y \|\beta e_1 - H_k y\|_2$, where $H_k = (h_{ij})_{i \leq k+1, j \leq k}$ and $e_1 = (1, 0, \ldots, 0)^T$ \label{step:gmres6}
\State $x_k = x_0 + V_k y_k$, where $V_k = (v_1, \ldots, v_k) \in \bbR^{m\times k}$  \label{step:gmres7}
\State \Return $x_k$
\end{algorithmic}
\end{algorithm}

The \emph{minimum residual method} (MINRES; \cite{paige1975solution}) can be understood as a special case of GMRES when $A$ is a symmetric matrix.
In this case, Arnoldi's method (steps \ref{step:gmres0}-\ref{step:gmres5}) in Algorithm \ref{alg:gmres} can be replaced by the simpler \emph{Lanczos algorithm} \cite{lanczos1950iteration}, described in Algorithm \ref{alg:lanczos}, where $\alpha_j = h_{jj}$ and $\beta_j = h_{j-1, j}$.

\begin{algorithm}
\caption{Lanczos algorithm} \label{alg:lanczos}
\begin{algorithmic}[1]
\State $\beta_1 \gets 0$
\State $v_0 \gets 0$
\For{$j=1, \ldots, k$} 
	\State $w_j \gets A v_j - \beta_j v_{j-1}$
	\State $\alpha_j \gets w_j^T v_j$
	\State $w_j \gets w_j - \alpha_j v_j$
	\State $\beta_{j+1} \gets \|w_j\|_2$
	\If {$\beta_{j+1} < \epsilon_{\rm tol}$ } set $k\gets j$ and {\bf break}\EndIf
	\State $v_{j+1} \gets w_j / \beta_{j+1}$
\EndFor
\end{algorithmic}
\end{algorithm}

In summary, standard approaches for solving \eqref{eq:ls} are (i) CG for SPD $A$; (ii) MINRES for symmetric $A$; and (iii) GMRES for general $A$.
However, convergence is not guaranteed for GMRES.
As an alternative, one can solve the normal equation
\be \label{eq:normal}
	A^T Ax = A^T b,
\ee
where iterative algorithms guarantee convergence.
However, this approach is often avoided in practice because the matrix $A^T A$ is less well conditioned than the original $A$; see Chapter 8 of \cite{saad2003iterative}.
There are a large number of other general approaches, and many of them are variations and extensions of Krylov subspace methods.
Each method has some appealing properties, but it is difficult in general to analyze them theoretically.
See Chapter 7 of \cite{saad2003iterative}.
Also, there are some algorithms which are devised to solve a structured linear system \cite{ho2012fast, golub2003solving, bostan2008solving}.
To the best of our knowledge, however, there is no efficient iterative algorithm that can solve an arbitrary sparse linear system.
In particular, the most popular, GMRES, often has quite strange convergence properties, see \cite{embree2003tortoise} and \cite{greenbaum1996any}, making the algorithm difficult to use in practice.

In this paper, we propose an iterative method which guarantees convergence for an arbitrary linear system.
Under the assumption that $A, b$ and $x^*$ are nonnegative, the basic algorithm is known in \cite{vardi1993image} as an EM algorithm with an infinite number of observations.
Although the EM algorithm satisfies certain monotonicity criteria, see \cite{dempster1977maximum}, a detailed convergence analysis is omitted in \cite{vardi1993image}.
Independently from \cite{vardi1993image}, \citet{walker2017iterative} studied the same algorithm viewing it as a Bayesian updating algorithm and provided the proof for convergence.
The innovation of this paper is to extend the algorithm to general linear systems where $A, b$ and $x^*$ are not necessarily nonnegative, and to provide more detailed convergence analysis.
In particular, our convergence results include inconsistent systems, \ie\ the linear system \eqref{eq:ls} has no solution.
In this case, it is shown that $(x_n)$ converges to a certain minimal Kullback--Leibler divergence point.

The algorithm is easy to implement and requires small storage.
The proposed algorithm can serve as a suitable alternative to the Krylov subspace methods.
The new algorithm and its theoretical properties are studied in Section \ref{sec:main}.
A comparison to existing methods is provided in Section \ref{sec:comparison} and concluding remarks are given in Section \ref{sec:discussion}.

\subsection*{Notation}
Every vector such as $b$ and $x$ are column vectors, and components are denoted with a subscript, \eg\ $b=(b_i)$.
Dots in subscripts present the summation in those indices, \ie\ $a_{\cdot j} = \sum_{i=1}^m a_{ij}$.
The $j$th column of $A$ is denoted by $a^{(j)}$.
For $X$, which may be a vector or a matrix, is said to be nonnegative (positive, resp.) and denoted $X \geq 0$ ($x > 0$, resp.) if each component of $X$ is nonnegative (positive, resp.).
The number of nonzero elements of $X$ is denoted $\cN_X$.

\section{An iterative algorithm with guaranteed convergence}
\label{sec:main}

\subsection{Algorithm for solving nonnegative systems}

Assume that $A$ is a nonsingular square matrix and $A, b$ and $x^*$ are nonnegative.
In this case, \citet{vardi1993image} and \citet{walker2017iterative} proposed the iterative algorithm
\be\label{eq:bayes-update}
	x_{n+1,j} = \frac{x_{n,j}}{a_{\cdot j}} \sum_{i} a_{ij} \frac{b_i}{b_{n,i}}, \quad n\geq 0,
\ee
where $b_n = (b_{n,i}) = Ax_n$ and $x_0 \geq 0$ is an initial guess.
To briefly introduce the main idea, assume that $b,x$ and $a^{(j)}$ are probability vectors, \ie\ a vector with non-negative entries that sum to one.

Now consider discrete random variables $I$ and $J$ whose joint distribution is given by
$$
	\pr(J=j) = x_j^* \quad {\rm and} \quad \pr(I=i | J=j) = a_{ij}.
$$
Then, the marginal probability of $I$ is 
$$
	\pr(I=i) = \sum_j \pr(I=i|J=j) \pr(J=j) = \sum_j x_j^* a_{ij} = b_i.
$$
Note that 
\be\label{eq:bayes-prob}
	\pr(J=j|I=i) = \frac{\pr(J=j) \pr(I=i|J=j)}{\sum_{j^\prime}\pr(J=j^\prime) \pr(I=i|J=j^\prime)} = \frac{x_j^* a_{ij}}{\sum_{j^\prime} a_{ij^\prime} x_{j^\prime}^*}
\ee
by Bayes theorem.

\citet{vardi1993image} constructed the iteration \eqref{eq:bayes-update} through an EM algorithm.
With known $A$ and $b$, consider the problem of estimating $x^*$ based on the observation $I_1, \ldots, I_N$, where $(I_k, J_k)_{1 \leq k \leq N}$ are i.i.d. copies of $(I,J)$.
Since we do not directly observe $J_1, \ldots, J_N$, a standard method to find a maximum likelihood estimator is the EM algorithm.
Let $N_{ij}$ be the number of $k$'s such that $(I_k, J_k) = (i,j)$.
Then, the complete log-likelihood is
$$
	L^c(x) = \sum_{i,j} N_{ij} \log a_{ij} + \sum_j N_{\cdot j} \log x_j,
$$
so we have
$$
	Q(x|x_n) \stackrel{\rm def}{=} \bbE_{x_n}[L^c(x) | I_1, \ldots, I_k]
	= C + \sum_j \bbE_{x_n}[N_{\cdot j}|I_1, \ldots, I_n] \log x_j,
$$
where $C$ does not depend on $x$.
Thus, the EM iteration $x_{n+1} = \argmax_x Q(x|x_n)$ is given as
$$
	x_{n+1,j} =  \frac{\bbE_{x_n}[N_{\cdot j}|I_1, \ldots, I_n]}{\sum_{j'}\bbE_{x_n}[N_{\cdot j'}|I_1, \ldots, I_n]}.
$$
Since
$$
	\bbE_{x_n}[N_{ij}|I_1, \ldots, I_n] = 
	\frac{x_{n,j} a_{ij}}{\sum_{j^\prime} a_{ij^\prime} x_{n, j^\prime}} N_{i\cdot}
$$
by \eqref{eq:bayes-prob}, we have
$$
	x_{n+1,j} = x_{n,j} \sum_i \frac{a_{ij}}{\sum_{j^\prime} a_{ij^\prime} x_{j^\prime}} \frac{N_{i\cdot}}{N_{\cdot\cdot}}.
$$
Note that $N_{i\cdot}/N_{\cdot\cdot} \rightarrow b_i$ almost surely as $N \rightarrow \infty$, reducing the iteration \eqref{eq:bayes-update}.
Therefore, \eqref{eq:bayes-update} can be interpreted as an EM algorithm with infinite number of observations.

\citet{walker2017iterative} viewed the iteration \eqref{eq:bayes-update} as Bayesian updating.
Given a prior $x_n$ and an observation $I$, the posterior update of $x_{n,j}$ is given by
$$
	\frac{x_{n,j} a_{Ij}}{\sum_{j'} a_{Ij'} x_{n,j'}}
$$
by \eqref{eq:bayes-prob}.
Since we do not have data, a natural choice is to use average update
$$
	x_{n+1,j} = \sum_i \frac{x_{n,j} a_{ij}}{\sum_{j'} a_{ij'} x_{n,j'}} b_i
$$
which is exactly \eqref{eq:bayes-update}.

The update \eqref{eq:bayes-update} can also be understood as a fixed-point iteration.
From the identity
\bean
	x_j^* = \pr(J=j) = \sum_i \pr(J=j | I=i) \pr(I=i) = x_j^* \sum_i \frac{a_{ij} b_i}{\sum_{j^\prime} a_{ij^\prime} x_{j^\prime}^*},
\eean
we consider an equation $\phi(x) = x$, where
$$\phi_j(x)=\sum_i \frac{a_{ij}b_i }{\sum_{j^\prime} a_{ij^\prime} x_{j^\prime}}$$
and $\phi(x)$ is the corresponding vector.
Then, it is not difficult to see that $x=x^*$ if and only if $\phi_j(x)=1$ for every $j$.
Thus, if the recursive update
$$x_{n+1}=x_n\circ \phi(x_n),$$
where $\circ$ denotes elementwise product, converges, it does so to $x^*$.

If $A,b,x \geq 0$ but some of $b,x$ and $a^{(j)}$'s are not probability vectors, we can easily rescale the problem as 
\be\label{eq:reformula}
	\widetilde A \widetilde x = \widetilde b
\ee
with the update \eqref{eq:bayes-update}, where $\widetilde A = (a_{ij}/a_{\cdot j})_{i,j\leq m}$, $\widetilde x = (x_j a_{\cdot j} / b_\cdot)_{j=1}^m$ and $\widetilde b = (b_i / b_\cdot)_{i=1}^m$.

Theorem \ref{thm:bayes-rate} assures the convergence of the update \eqref{eq:bayes-update} with geometric rate.
We need well-known bounds for probability metrics for the proof.
For $m$-dimensional vectors $u, v \geq 0$, define the \emph{Kullback--Leibler (KL) divergence} $D(u,v) = \sum_{i=1}^m u_i \log(u_i / v_i)$ and \emph{total variation} $V(u,v) = \sum_{i=1}^m |u_i - v_i|$.
In the definition of the KL divergence, we let $u_i \log(u_i/v_i)=0$ if $u_i=0$ and $D(u,v) = \infty$ if $u_i>0$ and $v_i=0$ for some $i$.
It is well-known that $D(u,v) \geq 0$ for every pair of probability vectors $(u,v)$, and equality holds if and only if $u=v$.
Let $\|\cdot\|_1$ denotes the $\ell_1$-operator norm (\ie\ maximum absolute column sum) of a matrix.

\begin{theorem}\label{thm:bayes-rate}
Assume that $A \geq 0$, $x^*, b > 0, x_0 > 0$ and $A$ is nonsingular.
Then, for $(x_n)$ defined by \eqref{eq:bayes-update}, there exists $N$ such that
$D(\widetilde x^*, \widetilde x_{n+1}) \leq (1-\delta) D(\widetilde x^*, \widetilde x_n)$ for all $n \geq N$, where $\widetilde x^* = (x^*_j a_{\cdot j}/b_\cdot)_{j=1}^m$ and $\widetilde x_n = (x_{n,j} a_{\cdot j}/b_\cdot)_{j=1}^m$ and
$$
	\delta = \frac{1}{3\|A^{-1}\|_1^2} \min_{1\leq j\leq m} \widetilde x^*_j.
$$
\end{theorem}
\begin{proof}
If some of $b, x^*$ and $a^{(j)}$'s are not probability vectors, we can reformulate the problem using \eqref{eq:reformula}.
Therefore, we may assume without loss of generality that $b, x^*$ and $a^{(j)}$'s are probability vectors.
For any $x_0 > 0$, it is easy to see that $x_n > 0$ and $\sum_{j=1}^m x_{n,j} = 1$ for every $n\geq 1$.
Thus, $b_n$ and $x_n$ are also probability vectors for every $n\geq 1$.
From \eqref{eq:bayes-update} we have 
\bean
	\log x_{n+1,j} = \log x_{n,j} + \log \sum_{i=1}^m \left( \frac{b_i}{b_{n,i}} a_{ij} \right)
	\geq \log x_{n,j} + \sum_{i=1}^m a_{ij} \log \left( \frac{b_i}{b_{n,i}} \right),
\eean
where the inequality holds by Jensen.
Therefore,
$$
	\sum_{j=1}^m x_j^* \log x_{n+1,j} \geq \sum_{j=1}^m x_j^* \log x_{n,j} + D(b, b_n).
$$
This implies that 
\be \label{eq:KL-inequality}
	D(x^*, x_{n+1}) \leq D(x^*, x_n) - D(b, b_n),
\ee
and $D(x^* x_n)$ converges, by the monotone convergence theorem.
Thus, $D(b, b_n) \rightarrow 0$, which in turn implies that $x_n \rightarrow x^*$.

Note that 
$$
	V(x^*, x_n) = V(A^{-1}b, A^{-1}b_n) \leq \|A^{-1}\|_1 V(b, b_n),
$$
where $\|\cdot\|_1$ denotes the $\ell_1$-operator norm (\ie maximum absolute column sum) of the matrix.
Therefore,
$$
	D(b, b_n) \geq \half V^2(b, b_n) \geq \frac{1}{2\|A^{-1}\|_1^2} V^2(x^*, x_n),
$$
where the first inequality holds by Pinsker's inequality (\cite{pinsker1964information, csiszar2011information}).
Since
\bean
	D(x^*, x_n) &=& \sum_{j=1}^m x^*_j \log \frac{x^*_j}{x_{n,j}} \leq \sum_{j=1}^m x^*_j \left( \frac{x^*_j}{x_{n,j}} - 1\right)
	= \sum_{j=1}^m \left(1 + \frac{x^*_j - x_{n,j}}{x_{n,j}} \right) (x^*_j - x_{n,j})
	\\
	&=& \sum_{j=1}^m \frac{(x^*_j - x_{n,j})^2}{x_{n,j}} 
	\leq \left(\sum_{j=1}^m \frac{|x^*_j - x_{n,j}|}{\sqrt{x_{n,j}}} \right)^2
	\leq V^2(x^*, x_n) \max_{1 \leq j \leq m} x_{n,j}^{-1}
\eean
and $x_n \rightarrow x^*$, we have $D(b, b_n) \geq \delta D(x^*, x_n)$ for all large enough $n$, where
$$
	\delta = \frac{1}{3\|A^{-1}\|_1^2} \min_{1\leq j\leq m} x^*_j.
$$
Therefore, by \eqref{eq:KL-inequality},
$$
	\delta D(x^*, x_n) \leq D(b, b_n) \leq D(x^*, x_n) - D(x^*, x_{n+1})
$$
for all large enough $n$.
It follows that $D(x^*, x_{n+1}) \leq (1-\delta) D(x^*, x_n)$ for all large $n$.
\end{proof}

\bigskip
Note that for any nonnegative vectors $p$ and $q$ with the same $\ell_1$-norm, the Kullback--Leibler  divergence and the Euclidean norm are related as
$$
	\frac{1}{\|p\|_1}\sum_{j} p_j\log (p_j/q_j) \geq \frac{\|p-q \|_1^2}{2 \|p\|_1^2}  \geq \frac{\|p-q \|_2^2}{2 \|p\|_1^2}.
$$
Thus, $\|x_n - x^*\|_2^2 \leq 2 \|x^*\|_1^2 D(\widetilde x^*, \widetilde x_n)$.

The key to the proof of Theorem \ref{thm:bayes-rate} is inequality \eqref{eq:KL-inequality}.
This inequality implies that the larger $D(b, b_n)$ is the larger we gain at the $n$th iteration.
It should be noted that $x^*, b > 0$ is essential for the convergence of the algorithm.
When $b_i \leq 0$ for some $i$, we can easily reformulate the problem as 
\be\label{eq:reformula2}
	A x_t = b_t,
\ee
where $x_t = x + t\bone_m, b_t = b + tA\bone_m, \bone_m = (1, \ldots, 1)^T$ and $t>0$ is a constant such that $b_t > 0$.
Note that $A \bone_m > 0$ because $A$ is nonsingular and nonnegative.
Note also that $A\geq 0$ and $b > 0$ does not imply that $x \geq 0$.
If $t$ is large enough, however, we have $x^* + t\bone_m > 0$, leading to Algorithm \ref{alg:bayes} which guarantees the convergence for any $A \geq 0$ and $b > 0$.
We call this algorithm as the \emph{nonnegative algorithm (NNA)}.
As seen in Section \ref{ssec:illustration}, $t$ can be chosen as a very large constant without being detrimental to the algorithm.

\begin{algorithm}
\caption{Nonnegative algorithm for $A \geq 0$ and $b > 0$} \label{alg:bayes}
\begin{algorithmic}[1]
\State {\bf Input}: $A, b, x_0, \epsilon_{\rm tol}$ and $t >0$
\State $b \gets b + t A\bone_m$
\State $x_0 \gets x_0 + t\bone_m$
\State $n \gets 0$
\While {$\|Ax_n - b\|_2 > \epsilon_{\rm tol}$ }
	\State $b_n \gets Ax_n$
	\State $c_n \gets b/b_n$ (componentwise division)
	\State $d_n = \widetilde{A}^T c_n$
	\State $x_{n+1} = d_n \circ x_n$ (componentwise multiplication)
	\State $n \gets n+1$
\EndWhile
\State $x_n \gets x_n - t\bone_m$
\State \Return $x_n$
\end{algorithmic}
\end{algorithm}

Here we consider the computational complexity of Algorithm \ref{alg:bayes}.
In \eqref{eq:bayes-update}, we first need to compute $b_n = Ax_n$, and then compute $c_n = b/b_n$, where $/$ represents componentwise division.
Finally, we compute $x_{n+1} = (\widetilde A^T c_n) \circ x_n$, where $\widetilde A = (a_{ij} / a_{\cdot j})_{i,j\leq m}$.
In summary, we need two matrix-vector multiplications and two vector-vector componentwise operations.
Assume that the sparsity structure of $A$ is known and $\cN_A \geq m$.
Then, the number of flops (floating-point operations; addition, subtraction, multiplication, or division) for matrix multiplication is less than $2 \cN_A$.
Also, for a vector-vector multiplication (or division), $2m$ flops are required.
Therefore, the total number of flops for one iteration of \eqref{eq:bayes-update} is less than $4(\cN_A + m)$.
We compare the number of flops with other algorithms in Section \ref{sec:comparison}.

We can apply Algorithm \ref{alg:bayes} for any linear system even when $A$ is not invertible or no solution exists.
For the remainder of this subsection, we assume that $A \in \bbR^{m_1\times m_2}$, $b \in \bbR^{m_2}$, $x \in \bbR^{m_1}$.

We first consider the case that a solution $x^*$ exists.
Since a solution may not be unique, it is not guaranteed that $x_n \rightarrow x^*$.
Theorem \ref{thm:consistent-system} assures the convergence of $b_n$ to $b$ with an upper bound of order $O(1/\epsilon)$ for the number of iterations to achieve $D(b_n, b) \leq \epsilon$.

\begin{theorem}\label{thm:consistent-system}
Assume that $A \geq 0$, $x^*, b > 0$ and $Ax^* = b$.
For any $x_0 > 0$, the sequence $(x_n)$ defined as \eqref{eq:bayes-update} satisfies $D(b, b_n) \rightarrow 0$.
In particular, for every $\epsilon > 0$ there exists $N \leq D(x^*, x_1) / \epsilon + 1$ such that $D(b_N, b) \leq \epsilon$.
\end{theorem}
\begin{proof}
As in the proof of Theorem \ref{thm:bayes-rate}, we may assume that $x^*, b$ and $a^{(j)}$, $1 \leq j \leq m_2$ are probability vectors without loss of generality.
Then, $b_n$ and $x_n$ are probability vectors for every $n \geq 1$, so the inequality \eqref{eq:KL-inequality} holds in the same way.
Thus, $D(x^*, x_n)$ converges by the monotone convergence theorem and
it follows that $D(b, b_n) \rightarrow 0$.

For a given $\epsilon > 0$, let $N$ be the largest integer less than or equal to $D(x^*, x_1) / \epsilon + 1$ and assume that $D(b, b_n) > \epsilon$ for every $n \leq N$.
Then, since 
$$
	0 \leq D(x^*, x_{N+1}) \leq D(x^*, x_1) - \sum_{n=1}^N D(b, b_n),
$$
using \eqref{eq:KL-inequality}, we have $N < D(x^*, x_1)/\epsilon$.
This makes a contradiction and completes the proof.
\end{proof}

Assume that the linear system \eqref{eq:ls} do not have a solution.
In this case, the iteration \eqref{eq:bayes-update} converges to a minimal KL divergence points as Theorem \ref{thm:am}.
For the proof, we view the iteration \eqref{eq:bayes-update} as an alternating minimization for which powerful tools have been developed in \cite{csiszar1984information} to study its convergence.

\begin{theorem} \label{thm:am}
Assume that $A \geq 0$ and $b > 0$.
For any $x_0 > 0$, the sequence $(x_n)$ defined as \eqref{eq:bayes-update} satisfies $\lim_n D(b, Ax_n) \downarrow \inf_x D(b, Ax)$, where $x$ ranges over every positive vector with $\sum_j x_j a_{\cdot j} = b_\cdot$.
\end{theorem}
\begin{proof}
Without loss of generality, we may assume that $b$ and $a^{(j)}$, $1 \leq j \leq m_2$ are probability vectors.
Let $\cP$ and $\cQ$ be the set of every bivariate probability mass functions $(i,j) \mapsto p(i,j)$ and $(i,j) \mapsto q(i,j)$ such that $\sum_j p(i,j) = b_i$ and $q(i,j) = a_{ij} x_j$ for some probability vector $x$, respectively.
Then, it is obvious that $\cP$ and $\cQ$ are convex.
Let
\bean
	q_n(i,j) = x_{n,j} a_{ij}
	\quad {\rm and} \quad
	p_n(i,j) = \frac{b_i q_n(i,j)}{\sum_{j'} q_n(i,j')}.
\eean
Then, $D(p_n, q_n) = D(b, b_n) \leq D(p, q_n)$ for every $p \in \cP$, where the inequality holds because $b$ and $b_n$ are marginal probabilities of $p$ and $q_n$.
Thus, $p_n = \argmin_{p \in \cP} D(p, q_n)$.

For a probability vector $x$ let $q(i,j) = a_{ij} x_j$.
Then,
\bean
	D(p_n, q) = \sum_{i,j} p_n(i,j) \log \frac{p_n(i,j)}{q(i,j)} = C_1 - \sum_{i,j} p_n(i,j) \log x_j
	\\
	= C_2 - \sum_j p_n(\cdot, j) \log \frac{x_j}{p_n(\cdot,j)},
\eean
where $p_n(\cdot, j) = \sum_i p_n(i,j)$ and $C_j$'s are terms independent of $x$.
Since
\bean
	p_n(\cdot, j) = \sum_i \frac{b_i x_{n,j} a_{ij}}{\sum_{j'} x_{n,j'} a_{ij'}} 
	= x_{n,j} \sum_i a_{ij} \frac{b_i}{b_{n,i}} = x_{n+1,j},
\eean
we have $D(p_n, q) = C_2 + D(x_{n+1}, x)$.
It follows that $q_{n+1} = \argmin_{q\in\cQ} D(p_n, q)$.

In summary, the sequences $(p_n)$ and $(q_n)$ are obtained by alternating minimization.
By Theorem 3 of \cite{csiszar1984information}, $D(p_n, q_n) \downarrow \inf_{p\in\cP, q\in\cQ} D(p, q)$.
Note that when 
\bean
	q(i,j) = a_{ij} x_j
	\quad {\rm and} \quad
	p(i,j) = \frac{b_i q(i,j)}{\sum_{j'} q(i,j')},
\eean
we have $D(p,q) = D(b, Ax)$.
Therefore, $D(p_n, q_n) = D(b, b_n) \downarrow \inf_x D(b, Ax)$.
\end{proof}

\subsection{General linear systems}

For convenience, we only consider a square matrix $A$, but the approach introduced in this subsection can also be applied to any linear system.
The main idea is to embed the original system \eqref{eq:ls} into a larger nonnegative system, and then apply Algorithm \ref{alg:bayes}.
This kind of slack variable techniques are well-known in linear algebra and optimization.
The enlarged system should be minimal to reduce any additional computational burden.

As an illustrative example, consider the system of linear equations
\bean 
	\left.\begin{array}{ccccccc}
	a_{11} x_1 &-& a_{12} x_2 &+& a_{13} x_3 &=& b_1,
	\\
	a_{21} x_1 &+& a_{22} x_2 &-& a_{23} x_3 &=& b_2,
	\\
	a_{31} x_1 &+& a_{32} x_2 &+& a_{33} x_3 &=& b_3,
	\end{array}\right.
\eean
where $a_{ij} \geq 0$ for every $i$ and $j$, so $A$ has negative elements.
We consider two more equations
\bean
	x_2 + x_4 = 0 \quad {\rm and} \quad x_3 + x_5 = 0,
\eean
where each equation contains only two nonzero elements.
Then, it is easy to see that solving the linear system consisting of the above five equations is equivalent to solving the following five equations:
\be\label{eq:ls3}
	\left.\begin{array}{rrrrrrrrrrr}
	a_{11} x_1	& &				&+& a_{13} x_3	&+& a_{12} x_4	& &				&=& b_1,
	\\
	a_{21} x_1	&+& a_{22} x_2	& & 			& &				&+& a_{23} x_5	&=& b_2,
	\\
	a_{31} x_1	&+& a_{32} x_2	&+& a_{33} x_3	& &				& &				&=& b_3,
	\\
				& &	x_2			& &				&+&	x_4			& &				&=& 0,
	\\
				& &				& & x_3			& &				&+& x_5			&=& 0.
	\end{array}\right.
\ee
Let $Py = c$ be the matrix form of \eqref{eq:ls3}, then we have $P \geq 0$, so NNA can be applied.
This can be generalized as in the following theorem.

\begin{theorem} \label{thm:consistent}
For $A \in \bbR^{m\times m}$ and $b \in \bbR^m$, assume that $Ax^* = b$.
Then, there exists a linear system $Py = c$ with solution $y^*$, such that $P$ is a $(m+J)\times (m+J)$ matrix with $J \leq m$, $\cN_P = \cN_A + 2J$ and the first $m$ components of $y^*$ are equal to $x^*$.
\end{theorem}
\begin{proof}
Let $\cJ = \{j \leq m: a_{ij} < 0 \;\textrm{for some $i\leq m$}\}$ and $J$ be the cardinality of $\cJ$.
If $J > 0$, we can write $\cJ = \{j_1, \ldots, j_J\}$ with $j_1 < \cdots < j_J$.
Let $A^+ = (\max\{a_{ij}, 0\})_{i,j\leq m}$, $A^- = -(\min\{a_{ij}, 0\})_{i,j\leq m}$ and $\widetilde A^-$ be the $m\times J$ sub-matrix of $A^-$ consisting of all nonzero columns.
Let $D = (d_{ij})$ be the $J \times m$ matrix defined as
$$
	d_{ij} = \left\{ \begin{array}{cc} 1 & \textrm{if $j=j_i$} \\ 0 & \textrm{otherwise.} \end{array}\right.
$$
Define a $(m+J) \times (m+J)$ matrix $P$ as
$$
	P = 
	\left(\begin{matrix}
		A^+ 	& \widetilde A^-
		\\
		D		& I_J
	\end{matrix}\right),
$$
where $I_J$ denotes the $J\times J$ identity matrices.
It is obvious that $\cN_P = \cN_A + 2J$.
Consider the linear system 
\be\label{eq:ls2}
	Py = c,
\ee
where $c = (b^T, {\bf 0}_J^T)^T$ and ${\bf 0}_J \in\bbR^J$ is the zero vector.
Then it is easy to see that $y^* = ((x^*)^T, -(x^*_\cJ)^T)^T$ is a solution of \eqref{eq:ls2}, where $x^*_\cJ = (x^*_j)_{j\in\cJ}$.
\end{proof}

\bigskip
Hence, from the proof, we see that both $P$ and $c$ are easy to find.
The corresponding algorithm is summarized in Algorithm \ref{alg:general}, where $J=m$ is assumed for simplification.

\begin{algorithm}
\caption{General algorithm} \label{alg:general}
\begin{algorithmic}[1]
\State {\bf Input}: $A, b, x_0, \epsilon_{\rm tol}$ and $t >0$
\State $P \gets {\bf 0}_{2m\times 2m}$
\For {$i=1, \ldots, m$} {}
	\For {$j=1, \ldots, m$} {}
		\If {$a_{ij} > 0$} {$p_{ij} = a_{ij}$}
		\ElsIf {$a_{ij} < 0$} $p_{i,m+j} = -a_{ij}$
		\EndIf
	\EndFor
	\State $p_{m+i,i} \gets 1$
	\State $p_{m+i, m+i} \gets 1$
\EndFor
\State $c \gets (b^T, {\bf 0}_m^T)^T$
\State $y_0 \gets (x_0^T, -x_0^T)^T$
\State $y \gets$ NNA($P, c, y_0, \epsilon_{\rm tol}, t$) (Algorithm \ref{alg:bayes})
\State \Return $(y_1, \ldots, y_m)^T$
\end{algorithmic}
\end{algorithm}

\subsection{Illustrations} \label{ssec:illustration}

Firstly, we illustrate the effect of $t$ with a small dimensional example.
We set $m=10$ and generate a matrix $A$ by sampling $a_{ij}$ independently from the uniform distribution on the unit interval $[0,1]$.
Hence, with probability one, $A$ will be invertible.
Each component $b_i$ is also generated from the uniform distribution.
We then ran 100 iterations of Algorithm \ref{alg:bayes} with $t=10, 100$ and $1000$.
At each step, we obtain $\|Ax_n - b\|_2$, which are drawn in Figure \ref{fig:t-effect} with natural logarithmic scale.
The results are robust to the value of $t$, which is a common phenomenon with all our experiments.
Therefore, we can choose $t$ sufficiently large in practice.

\begin{figure} \bc
\includegraphics[width=100 mm, height=100 mm]%
  {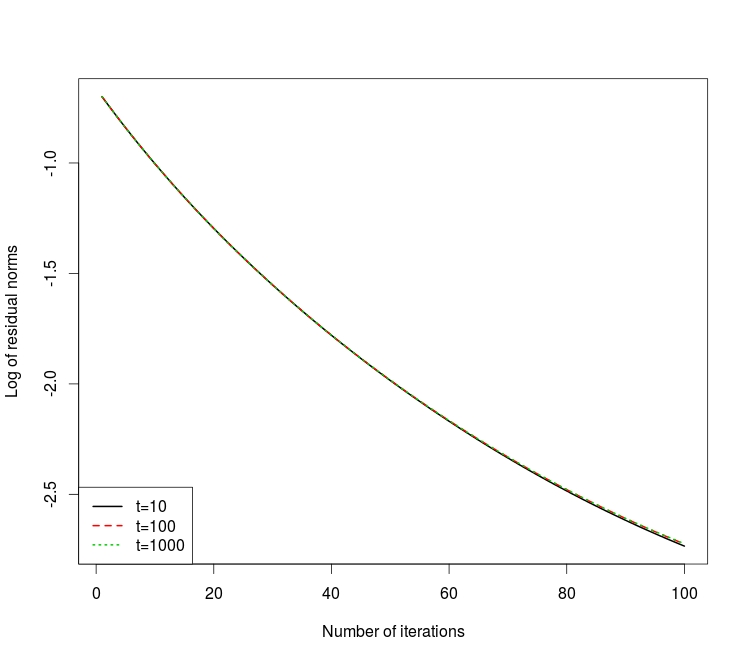}
\caption{The effect of the value of $t$. Residual norms $\|Ax_n - b\|_2$ are plotted on log scale for $t=10$ (black solid), $100$ (red dashed) and $1000$ (green dotted). The three lines are almost overlapped.}
\label{fig:t-effect}
\ec \end{figure}

We next consider a large sparse random matrix.
We set $m=1000$ and randomly generated $5m$ nonzero nondiagonal elements from the uniform distribution on $[0,1]$.
Each diagonal element of $A$ is generated from the uniform distribution on the interval $[0,100]$.
We compare the NNA algorithm ($t=0$) with GMRES, applied to the original system, and normal equation given in \eqref{eq:normal}.
The result is given in Figure \ref{fig:comparison}, showing the better convergence for the NNA compared to GMRES.

\begin{figure} \bc
\includegraphics[width=100 mm, height=100 mm]%
  {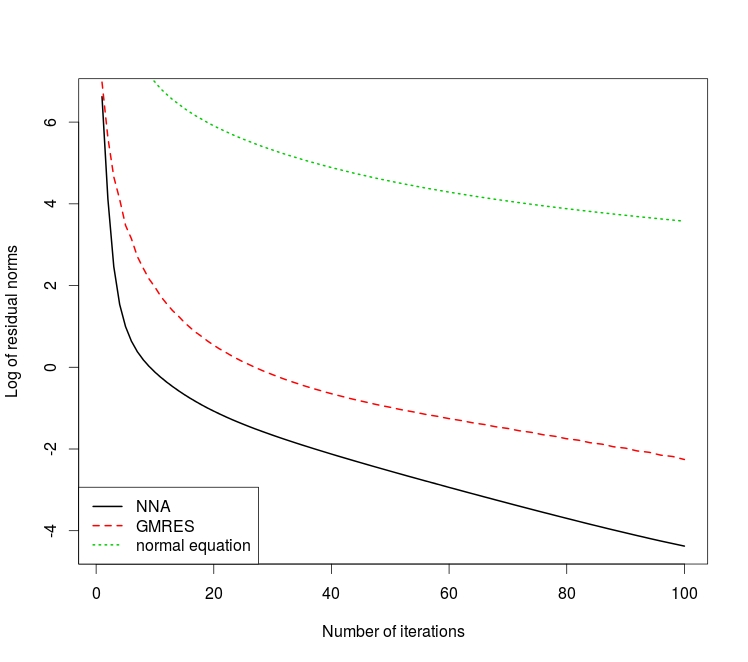}
\caption{Residual norms $\|A_s x_n - b\|_2$ are plotted in log scale for NNA (black solid), GMRES applied to the original system (red dashed) and GMRES to the normal equation (green dotted).}
\label{fig:comparison}
\ec \end{figure}

Finally, we consider a real example, known as GRE-1107, which can be found in \cite{matrixmarket}.
It is a nonsymmetric indefinite matrix, and the number of non-zero components is 5664 with $m=1107$.
NNA converges quickly without preconditioning while preconditioned (by the incomplete LU decomposition) GMRES and BI-CGSTAB \cite{van1992bi} fail to converge.
Residual norms are plotted in Figure \ref{fig:real}.

\begin{figure} \bc
\includegraphics[width=100 mm, height=100 mm]%
  {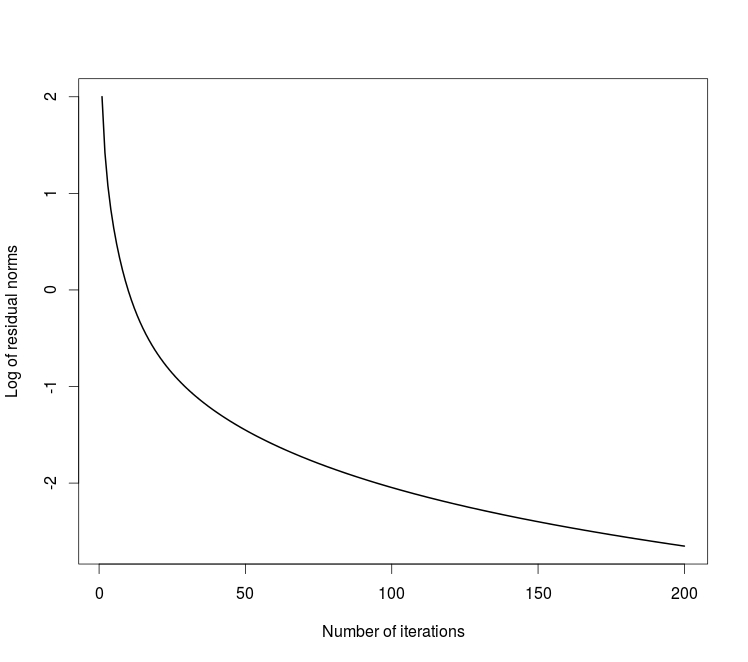}
\caption{Residual plot of NNA for the matrix GRE-1107.}
\label{fig:real}
\ec \end{figure}

In the next section we compare our algorithm with those mentioned in Section \ref{sec:intro}.
We do this under the conditions of guaranteed convergence, which impose a restriction on all algorithms, save our own.
In particular, we will compare flops per iteration and convergence rate.

\section{Comparison with other iterative methods} \label{sec:comparison}

As mentioned in the introduction, there is a vast amount of literature for solving sparse linear systems, but difficult to study theoretically.
As a consequence, only a few algorithms possess convergence properties, but even then under restrictive conditions.
In this section, we compare widely used iterative methods and their convergence properties.
Under the assumption that the arithmetic is exact, the result of this section is summarized in Table \ref{tab:comparison}.
Note that the computational complexities of MINRES$(k)$ and GMRES$(k)$ are not directly comparable to those of other methods because they depend on the number of step size $k$.

\begin{table}
\caption{Comparison of iterative methods with known convergence properties. The second column represents sufficient conditions guaranteeing the convergence: DD (diagonally dominant), PD (positive definite) and SPD (symmetric and PD).}
\label{tab:comparison}	
\bc
\begin{tabular}{c|ccc}
	\hline
	& Conditions for	& \multirow{2}{*}{FLOPs}		& \multirow{2}{*}{Storage} \\
	& convergence 		& 		&  \\
	\hline
	\multirow{2}{*}{NNA}		&  \multirow{2}{*}{-}		& \multirow{2}{*}{$O(\cN_A + m)$}	& \multirow{2}{*}{$O(\cN_A + m)$} \\
				& \\
	\hline
	\multirow{2}{*}{Jacobi}		& \multirow{2}{*}{DD}	& \multirow{2}{*}{$O(\cN_A + m)$}	& \multirow{2}{*}{$O(\cN_A + m)$} \\
								& &&\\
	\hline
	Gauss-		& \multirow{2}{*}{DD or SPD}	& \multirow{2}{*}{$O(\cN_A + m)$}	& \multirow{2}{*}{$O(\cN_A + m)$} \\
	Seidel		&&&\\
	\hline
	Conjugate	& \multirow{2}{*}{SPD} & \multirow{2}{*}{$O(\cN_A + m)$}	& \multirow{2}{*}{$O(\cN_A + m)$} \\
	gradient	&&&  \\
	\hline
	\multirow{2}{*}{MINRES$(k)$}		& \multirow{2}{*}{symmetric} & \multirow{2}{*}{$O(k \cN_A + k m)$}	& \multirow{2}{*}{$O(\cN_A + k m)$} \\
									&&& \\
	\hline
	\multirow{2}{*}{GMRES$(k)$}		& \multirow{2}{*}{PD} & \multirow{2}{*}{$O(k \cN_A + k^2 m)$}	& \multirow{2}{*}{$O(\cN_A + km)$} \\
									&&& \\
	\hline
\end{tabular}
\ec
\end{table}

\subsection{Basic methods: Jacobi and Gauss--Seidel}

It is easy to see that the numbers of flops for each step of the Jacobi and Gauss--Seidel methods are $2 (\cN_A + m)$.
Also, required storages is $\cN_A + 3m$ for Jacobi and $\cN_A + 2m$ for Gauss--Seidel.
Let $L, U$ and $D$ be the lower, upper triangular and diagonal parts of $A = L + U + D$, respectively.
Then, the Jacobi and Gauss--Seidel methods can be expressed in matrix forms as
$$
	x_{n+1} = D^{-1} \{ b - (L+U) x_n\} \quad \textrm{and} \quad
	x_{n+1} = (L+D)^{-1} (b-U x_n),
$$
respectively.
It is well-known (see Chapter 4 of \cite{saad2003iterative}) that updates of the form $x_{n+1} = G x_n + f$ for some $G\in\bbR^{m\times m}$ and $f\in\bbR^m$ assures convergence if $\rho(G) < 1$, where $\rho(G)$ is the spectral radius of $G$.
More specifically, $x_n$ obtained by the Jacobi and Gauss--Seidel methods satisfy 
$$
	\|x_n - x^*\|_2 \leq \{\rho(D^{-1} (L+U))\}^n \|x_0 - x^*\|_2
$$
and
$$
	\|x_n - x^*\|_2 \leq \{\rho((L+D)^{-1} U)\}^n \|x_0 - x^*\|_2,
$$
respectively.
It follows that $x_n \rightarrow x^*$ if the corresponding spectral radius is strictly smaller than 1.
For both methods, $x_n$ sometimes converges to $x^*$ even when the spectral radius is larger than 1.

It can be expensive to compute the spectral radius of a given large matrix.
Fortunately, there are well-known sufficient conditions which are easy to check.
A matrix $A\in\bbR^{m\times m}$ is called \emph{diagonally dominant} if $|a_{jj}| \geq \sum_{i\neq j} a_{ji}$ for every $i \geq 1$, and \emph{strictly diagonally dominant} if every inequality is strict.
A matrix $A$ is called \emph{irreducible} if the graph representation of $A$ is irreducible,
and \emph{irreducibly diagonally dominant} if it is irreducible, diagonally dominant and $|a_{jj}| > \sum_{i\neq j} |a_{ji}|$ for some $j \geq 1$.
If $A$ is strictly or irreducibly diagonally dominant, then $\rho(D^{-1} (L+U)) < 1$ and
$\rho((L+D)^{-1} U) < 1$; see Chapter 4 of \cite{saad2003iterative}.
Another sufficient condition for $\rho((L+D)^{-1} U) < 1$ is that $A$ is symmetric and positive definite; see \cite{golub2012matrix}.

\subsection{Conjugate gradient method}

It is easy to see that the number of flops in steps \ref{step:cg1}--\ref{step:cg2} of Algorithm \ref{alg:cg} is $2 \cN_A + 12m$, and the required storage is $\cN_A + 4m$.
Let $x_n$ be the approximate solution obtained at the $n$th step of the conjugate gradient method.
If the arithmetic is exact, we have $x_m = x^*$, so the exact solution can be found in $m$ steps.
If $m$ is prohibitively large, let $\lambda_{\max}(A)$ and $\lambda_{\min}(A)$ be the maximum and minimum eigenvalues of $A$, respectively.
Then, an upper bound on the conjugate norm between $x_n$ and $x^*$ is given as
$$
	(x_n - x^*)^T A (x_n - x^*) \leq 4 \left( \frac{\sqrt{\kappa}-1}{\sqrt{\kappa}+1}\right)^{2n} (x_0 - x^*)^T A (x_0 - x^*),
$$
where $\kappa = \lambda_{\max}(A) / \lambda_{\min}(A)$ (Chapter 6 of \cite{saad2003iterative}).
In practice, the improvement is typically linear in the step size; see \cite{liesen2004convergence}.

\subsection{MINRES and GMRES}

Ignoring the computational complexity of step \ref{step:gmres6}, that is relatively small for $k \ll m$, the numbers of flops required for steps \ref{step:gmres1}, \ref{step:gmres2}, \ref{step:gmres3}, \ref{step:gmres4}, \ref{step:gmres5} and \ref{step:gmres7} of Algorithm \ref{alg:gmres} are $2N_A, 2m, 2m, 2m, m$ and $(2k+1)m$, respectively.
Thus, the number of flops is $2k \cN_A + (2k^2 + 7k +1)m$.
Since we only need to save $A$, the orthonormal matrix $V_k \in \bbR^{m\times k}$, the approximate solution and vector for $A v_i$, the required storage is $\cN_A + (k+2)m$.
Here, storage for the Hassenberg matrix $H_k$ is ignored because $k$ is relatively small.
For the Lanczos algorithm (Algorithm \ref{alg:lanczos}), it is not difficult to see that the number of flops is $k(2\cN_A + 9m)$.

In general, Algorithm \ref{alg:gmres} does not guarantee convergence unless $k=m$.
In particular, it is shown in \cite{greenbaum1996any} that for any decreasing sequence $\epsilon_0 > \epsilon_1 > \cdots > \epsilon_m = 0$, there exists a matrix $A \in \bbR^{m\times m}$ and vectors $b, x_0 \in \bbR^m$ such that $\|G_k(x_0)\|_2 = \epsilon_k$.
Define $(x_n)$ as $x_{n+1} = G_k(x_n)$, a sequence generated by the restarted GMRES.
Then, if $A$ is positive definite, $x_n$ converges for any $k \geq 1$; see \cite{eisenstat1983variational}.
In particular, the rate is given by
$$
	\|Ax_n - b\|_2^2 \leq \left\{ 1 - \frac{\lambda_{\min}^2( (A+A^T)/2)}{\lambda_{\max} (A^T A)} \right\}^{nk} \|A x_0 - b\|_2^2.
$$
Some other convergence criteria of GMRES can be found in \cite{chronopoulos1991s}.
Also, more general upper bounds for residual norms, but not guaranteeing convergence, can be found in \cite{liesen2004convergence} and \cite{saad2003iterative}.

If $A$ is symmetric (not necessarily positive definite),
$$
	\|Ax_n - b\|_2^2 \leq \left\{ 1 - \frac{\lambda_{\min}^2(A^2)}{\lambda_{\max}(A^4)} \right\}^n \|Ax_0 - b\|_2^2
$$
for every $k \geq 2$; see \cite{chronopoulos1991s}, assuring the convergence of restarted MINRES.
Under a certain condition on the spectrum of $A$, a different type of upper bound can be found in \cite{liesen2004convergence}.

\subsection{$s$-step methods}

A number of $s$-step methods and their convergence properties are studied in \cite{chronopoulos1991s}.
In particular, it is shown that $s$-step generalized conjugate residual, Orthomin$(k)$ and minimal residual methods converge for all positive definite and some indefinite matrices.
Here, $s$-step minimal residual method is mathematically equivalent to GMRES$(s)$.
However, it is not easy in practice to check conditions for convergence of indefinite matrices.
Furthermore, computational costs for $s$-step methods can be expensive because they require more matrix-vector multiplications in each step.



\section{Discussion}
\label{sec:discussion}

The main contribution of the paper is to describe an algorithm which guarantees convergence for indefinite linear systems of equations.
The key idea is that arbitrary systems  can be embedded within a nonnegative system.
Other algorithms, such as CG and GMRES$(k)$, guarantee convergence under certain conditions, but it is difficult in general to transform an arbitrary system into a guaranteed convergent one for them.

Finally, we could do the updates using parallel computing which would provide faster convergence times.

\section*{Acknowledgement}
The second author is partially supported by NSF grant DMS No.\ 1612891.

\bibliographystyle{apalike}
\bibliography{bibliography}   

\begin{thebibliography}{}

\bibitem[Arnoldi, 1951]{arnoldi1951principle}
Arnoldi, W.~E. (1951).
\newblock The principle of minimized iterations in the solution of the matrix
  eigenvalue problem.
\newblock {\em Quarterly of Applied Mathematics}, 9(1):17--29.

\bibitem[Aune et~al., 2013]{aune2013iterative}
Aune, E., Eidsvik, J., and Pokern, Y. (2013).
\newblock Iterative numerical methods for sampling from high dimensional
  {G}aussian distributions.
\newblock {\em Statistics and Computing}, 23(4):501--521.

\bibitem[Axelsson, 1980]{axelsson1980conjugate}
Axelsson, O. (1980).
\newblock Conjugate gradient type methods for unsymmetric and inconsistent
  systems of linear equations.
\newblock {\em Linear Algebra and Its Applications}, 29:1--16.

\bibitem[Bostan et~al., 2008]{bostan2008solving}
Bostan, A., Jeannerod, C.-P., and Schost, {\'E}. (2008).
\newblock Solving structured linear systems with large displacement rank.
\newblock {\em Theoretical Computer Science}, 407(1):155--181.

\bibitem[Chronopoulos, 1991]{chronopoulos1991s}
Chronopoulos, A.~T. (1991).
\newblock $s$-step iterative methods for (non) symmetric (in) definite linear
  systems.
\newblock {\em SIAM Journal on Numerical Analysis}, 28(6):1776--1789.

\bibitem[Csiszar and K{\"o}rner, 2011]{csiszar2011information}
Csiszar, I. and K{\"o}rner, J. (2011).
\newblock {\em Information Theory: Coding Theorems for Discrete Memoryless
  Systems}.
\newblock Cambridge University Press.

\bibitem[Csiszz{\'a}r and Tusn{\'a}dy, 1984]{csiszar1984information}
Csiszz{\'a}r, I. and Tusn{\'a}dy, G. (1984).
\newblock Information geometry and alternating minimization procedures.
\newblock {\em Statistics \& Decisions, {\rm Supplemental Issue No. 1}}, pages
  205--237.

\bibitem[Dempster et~al., 1977]{dempster1977maximum}
Dempster, A.~P., Laird, N.~M., and Rubin, D.~B. (1977).
\newblock Maximum likelihood from incomplete data via the em algorithm.
\newblock {\em Journal of the royal statistical society. Series B
  (methodological)}, 39(1):1--38.

\bibitem[Eisenstat et~al., 1983]{eisenstat1983variational}
Eisenstat, S.~C., Elman, H.~C., and Schultz, M.~H. (1983).
\newblock Variational iterative methods for nonsymmetric systems of linear
  equations.
\newblock {\em SIAM Journal on Numerical Analysis}, 20(2):345--357.

\bibitem[Elman, 1982]{elman1982iterative}
Elman, H.~C. (1982).
\newblock {\em Iterative Methods for Large, Sparse, Nonsymmetric Systems of
  Linear Equations}.
\newblock PhD thesis, Yale University.

\bibitem[Embree, 2003]{embree2003tortoise}
Embree, M. (2003).
\newblock The tortoise and the hare restart {GMRES}.
\newblock {\em SIAM Review}, 45(2):259--266.

\bibitem[Golub and Greif, 2003]{golub2003solving}
Golub, G.~H. and Greif, C. (2003).
\newblock On solving block-structured indefinite linear systems.
\newblock {\em SIAM Journal on Scientific Computing}, 24(6):2076--2092.

\bibitem[Golub and Van~Loan, 2012]{golub2012matrix}
Golub, G.~H. and Van~Loan, C.~F. (2012).
\newblock {\em Matrix Computations}.
\newblock Johns Hopkins University Press, 3rd edition.

\bibitem[Greenbaum et~al., 1996]{greenbaum1996any}
Greenbaum, A., Pt{\'a}k, V., and Strako{\v{s}}, Z. (1996).
\newblock Any nonincreasing convergence curve is possible for {GMRES}.
\newblock {\em SIAM Journal on Matrix Analysis and Applications},
  17(3):465--469.

\bibitem[Hestenes and Stiefel, 1952]{hestenes1952methods}
Hestenes, M.~R. and Stiefel, E. (1952).
\newblock Methods of conjugate gradients for solving linear systems.
\newblock {\em Journal of Research of the National Bureau of Standards},
  49(6):409--436.

\bibitem[Ho and Greengard, 2012]{ho2012fast}
Ho, K.~L. and Greengard, L. (2012).
\newblock A fast direct solver for structured linear systems by recursive
  skeletonization.
\newblock {\em SIAM Journal on Scientific Computing}, 34(5):A2507--A2532.

\bibitem[Kennedy and Gentle, 1980]{kennedy1980gentle}
Kennedy, W.~G. and Gentle, J.~E. (1980).
\newblock {\em Statistical Computing}.
\newblock Dekker, New York.

\bibitem[Koenker and Ng, 2003]{koenker2003sparsem}
Koenker, R. and Ng, P. (2003).
\newblock {SparseM: A sparse matrix package for R}.
\newblock {\em Journal of Statistical Software}, 8(6):1--9.

\bibitem[Lanczos, 1950]{lanczos1950iteration}
Lanczos, C. (1950).
\newblock An iteration method for the solution of the eigenvalue problem of
  linear differential and integral operators.
\newblock {\em Journal of Research of the National Bureau of Standards},
  45(4):255--282.

\bibitem[Liesen and Tich{\`y}, 2004]{liesen2004convergence}
Liesen, J. and Tich{\`y}, P. (2004).
\newblock Convergence analysis of {K}rylov subspace methods.
\newblock {\em GAMM-Mitteilungen}, 27(2):153--173.

\bibitem[{Matrix Market}, 2007]{matrixmarket}
{Matrix Market} (2007).
\newblock {N}ational {I}nstitute of {S}tandards and {T}echnology.
  \url{http://math.nist.gov/MatrixMarket}.

\bibitem[Paige and Saunders, 1975]{paige1975solution}
Paige, C.~C. and Saunders, M.~A. (1975).
\newblock Solution of sparse indefinite systems of linear equations.
\newblock {\em SIAM Journal on Numerical Analysis}, 12(4):617--629.

\bibitem[Pinsker, 1964]{pinsker1964information}
Pinsker, M. (1964).
\newblock {\em Information and Information Stability of Random Variables and
  Processes}.
\newblock Holden-Day, San Francisco.

\bibitem[Saad, 2003]{saad2003iterative}
Saad, Y. (2003).
\newblock {\em Iterative Methods for Sparse Linear Systems}.
\newblock SIAM, 2nd edition.

\bibitem[Saad and Schultz, 1985]{saad1985conjugate}
Saad, Y. and Schultz, M.~H. (1985).
\newblock Conjugate gradient-like algorithms for solving nonsymmetric linear
  systems.
\newblock {\em Mathematics of Computation}, 44(170):417--424.

\bibitem[Saad and Schultz, 1986]{saad1986gmres}
Saad, Y. and Schultz, M.~H. (1986).
\newblock {GMRES}: A generalized minimal residual algorithm for solving
  nonsymmetric linear systems.
\newblock {\em SIAM Journal on Scientific and Statistical Computing},
  7(3):856--869.

\bibitem[Van~der Vorst, 1992]{van1992bi}
Van~der Vorst, H.~A. (1992).
\newblock {Bi-CGSTAB}: {A} fast and smoothly converging variant of {Bi-CG} for
  the solution of nonsymmetric linear systems.
\newblock {\em SIAM Journal on Scientific and Statistical Computing},
  13(2):631--644.

\bibitem[Vardi and Lee, 1993]{vardi1993image}
Vardi, Y. and Lee, D. (1993).
\newblock From image deblurring to optimal investments: Maximum likelihood
  solutions for positive linear inverse problems.
\newblock {\em Journal of the Royal Statistical Society. Series B
  (Methodological)}, 55(3):569--612.

\bibitem[Walker, 2017]{walker2017iterative}
Walker, S.~G. (2017).
\newblock An iterative algorithm for solving sparse linear equations.
\newblock {\em Communications in Statistics-Simulation and Computation},
  46(7):5113--5122.

\bibitem[Young and Jea, 1980]{young1980generalized}
Young, D.~M. and Jea, K.~C. (1980).
\newblock Generalized conjugate-gradient acceleration of nonsymmetrizable
  iterative methods.
\newblock {\em Linear Algebra and Its Applications}, 34:159--194.

\end{thebibliography}

\end{document}